\documentclass[a4paper,11pt]{article}
\usepackage{amsfonts,amsmath,amssymb,amsthm,esint}

\newtheorem{lemma}{Lemma}[section]
\newtheorem{theorem}[lemma]{Theorem}
\newtheorem{corollary}[lemma]{Corollary}
\newtheorem{proposition}[lemma]{Proposition}
\newtheorem{remark}[lemma]{Remark}

\def\C{\widetilde C}

\def\Om{\Omega}

\def\W{\widetilde{M}}
\def\WW{\widetilde{W}}
\def\w{\widetilde{w}}
\newcommand{\mc}{\mathcal}

\newcommand{\opsi}{{\overline \psi}}

\newcommand{\oc}{{\overline c}}

\newcommand{\tr}{\operatorname{\text{tr}}}
\newcommand{\Z}{{\mathbb Z}}

\newcommand{\R}{{\mathbb R}}
\newcommand{\Per}{{\rm Per}}
\newcommand{\T}{{\mathbb T}}

\begin{document}

\title{Long-time behavior of the mean curvature flow with periodic forcing}

\author{
Annalisa Cesaroni\footnote{
Dipartimento di Matematica Pura e Applicata,
Universit\`a di Padova,
via Trieste 63, 35121 Padova, Italy, email: acesar@math.unipd.it, novaga@math.unipd.it}  
\and Matteo Novaga\footnotemark[\value{footnote}]
}

\date{}

\maketitle

\begin{abstract}
We consider the long-time behavior of the mean curvature flow 
in heterogeneous media with periodic fibrations, modeled as an additive driving force. 
Under appropriate assumptions on the forcing term, we show existence 
of generalized traveling waves with maximal speed of propagation,  
and we prove the convergence of solutions to the forced mean curvature flow 
to these generalized waves. 
\end{abstract}
 
 
\section{Introduction}

We are interested in  the long-time behavior of the mean curvature flow  in a periodic heterogeneous medium. 
The evolution  law can be written as a forced mean curvature flow
\[
v= \kappa -g 
\] 
where $v$ denotes the inward normal velocity of the evolving hypersurface, $\kappa$ its mean curvature 
(with the convention that $\kappa$ is positive on convex sets)
and $g$ is a periodic forcing term. 
In our model, we assume that the hypersurfaces are graphs with respect to 
a fixed hyperplane  and that the forcing term $g$ does not depend on the variable orthogonal to such  hyperplane (fibered medium). 
Under these assumptions the evolving hypersurface coincides with
the graph of the solution   to the Cauchy problem
\begin{equation}\label{graphequation} 
\left\{ \begin{array}{ll} 
u_t   =    \sqrt{1+|Du |^2}\,{\rm div}\left( \dfrac{Du }{\sqrt{1+|Du|^2}}\right)+  g  \sqrt{1+|Du |^2} 
&   \text{in } (0, +\infty)\times \R^n \\ u(0,\cdot)  = u_0  &  \text{in }  \R^n.
\end{array}\right.
\end{equation}  
We are particularly interested in the asymptotic behavior as $t\to +\infty$ of solutions 
to \eqref{graphequation}, where the initial data $u_0$ and the forcing term $g$ are assumed  
to be Lipschitz continuous and $\Z^n$-periodic. 

The expected result is that, under appropriate assumptions on $g$, there exists a unique constant 
$c\in \R$ and a periodic function $\psi$ such that 
\[u(t,y)-ct -\psi(y) \to 0, \qquad  \text{  as }t\to +\infty,\text{ uniformly in }\R^n .\]  
This is a  result on the asymptotic stability of special solutions to \eqref{graphequation}, 
called  traveling wave solutions, which are of the form $\psi+ct$.  The constant $c$ and the function $\psi$ are respectively the   propagation speed and  the profile of the wave. 

The first question we address in Section \ref{secex} of this paper is about existence of traveling wave solutions to \eqref{graphequation}. 
We provide a construction of such solutions using a  variational approach developed in  \cite{mu} (see also \cite{nm2}).  
In particular, our solutions are critical points of appropriate functionals, 
which are exponentially weighted area functionals with a volume term, depending on the speed of propagation $c$. 
Exploiting this variational structure, we show existence of traveling waves 
under rather weak assumptions on the forcing term $g$, i.e.    
\[
\exists\, A\subseteq (0,1)^n \text{ s.t. }   
\int_{A} g(y)\,dy\, > \,\Per(A,\mathbb T^n)     
\]  where $\Per(A,\mathbb T^n)$ is the periodic perimeter of $A$ (see Section \ref{secnot}). 
Notice that, if $\int_{(0,1)^n} g>0$, then 
the previous condition holds true by taking $A=(0,1)^n$. 

As our solutions are in general not globally defined, we call them {\em generalized traveling waves}. 
In Propositions \ref{psireg} and \ref{support}
we discuss the regularity of these solutions and of their support. 
Moreover, 
in Section \ref{globalsolutionconditions} we list some stronger conditions on the forcing term, involving 
only the oscillation and the norm of $g$, 
under which we show existence of classical (i.e. globally defined) traveling waves (see Proposition \ref{classicaltw}).

We point out that the variational method selects the {\em fastest} traveling waves for \eqref{graphequation} which are bounded above,
in particular it is uniquely defined the speed of propagation $\oc$ of such waves and it holds $\oc\ge \int_{(0,1)^n} g$
(see Corollary \ref{coro}).

We recall that the problem of existence of classical traveling waves for the forced mean curvature flow has already been considered in the literature,
under different assumptions on the forcing term \cite{ls,dky,CLS}.
We also mention \cite{mrr}, where the authors construct  
$V$-shaped traveling waves in the whole space for a constant forcing term (see also \cite{nt,bc,lb} for similar results in the planar case). 
The construction of the traveling fronts in these papers relies mainly on maximum principle type arguments, 
while we use here a variational approach. 
 
The second question of interest is about the convergence, as $t\to+\infty$, of the solution  to \eqref{graphequation} to a traveling wave solution.  
We point out that the long-time behavior of solutions of parabolic problems using viscosity solutions type arguments has been extensively considered in the literature: see \cite{nr} and \cite{bs} for the case of semilinear and quasilinear parabolic problems in periodic environments,  \cite{dl} where the author considers uniformly parabolic operators in bounded domains with Neumann boundary conditions, and \cite{bpt} for the case of  viscous Hamilton-Jacobi equations in bounded domains with Dirichlet boundary conditions. 
However, none of these results applies to mean curvature type equations such as \eqref{graphequation}.   

In Section \ref{secconv} we prove a convergence result under the assumption that there exists a global traveling wave solution. 
In particular, in Corollary \ref{convfort} we show that  the solution $u(t,y)$ to \eqref{graphequation} satisfies
\[
u(t,y)-\oc t \to \psi(y) \qquad {\rm in\ } \mc C^{1+\alpha}(\R^n),{\rm\ as\ }t\to+\infty,
\]
where $\psi + \oc t$ is the traveling wave, which in this case is unique up to an additive constant.

In the general case, we obtain a weaker convergence result. First,
in Proposition \ref{meanconv} we describe the asymptotic behavior as $t\to +\infty$ of the maximum of the function $u(t, \cdot)$.
Namely, letting $Q:=(0,1)^n$, we show that there exists a constant $K >0$ such that 
\[
\min_Q u_0+\oc t \leq \max_Q u(t,y) \leq  \oc t +  K +\frac{\log (1+t)}{\oc} . 
\] 
Then, in Theorem \ref{convergence2} we show that, along a subsequence $t_n\to +\infty$, 
\[
u(t_{n },y)-\max_Q  u(t_{n },\cdot)  \longrightarrow 
\left\{\begin{array}{ll}
\psi(y)  & {\rm\ locally\ in\ } \mc C^{1+\alpha}(E )
\\
-\infty & {\rm\ locally\ uniformly\ in\ } Q\setminus \overline E 
\end{array}\right.
\]  
for all $\alpha\in (0,1)$, where $\psi +\oc t$ is a generalized traveling wave supported in $E\subset Q$.

We point out that the proof of the convergence result, as well as the proof of existence of generalized waves, 
essentially uses variational methods, rather than maximum principle based argumets.

\smallskip 

\paragraph{Acknowledgements.}
The authors warmly thank Guy Barles and Cyrill Muratov for inspiring discussions on this problem.
 
\section{Notation and preliminary results}\label{secnot}

We refer to \cite{AFP} for a general introduction to functions of bounded variation and sets of finite perimeter.
Letting $Q:=(0,1)^n$,
it is a classical result that any  $u\in BV(Q)$ admits a  trace $u^Q $ on $\partial Q$ 
(see e.g. \cite[Thm. 3.87]{AFP}).  Let $\partial_0 Q := \partial Q\cap \{y: \prod_{i=1}^n y_i =0\}$ and let 
$\sigma:  \partial_0 Q\to \partial Q$  be the function 
$\sigma(y) := y+\sum_{i=1}^n \lambda_i(y)e_i$, where $\lambda_i(y)=1$ if $y_i=0$ and $\lambda_i(y)=0$ otherwise.  

We consider the space $BV_{\rm per}(Q)$ of functions 
which have periodic bounded variation in $Q$, where 
the periodic total variation of
$u\in BV(Q)$ is defined as 
\begin{equation}\label{bvper}
|Du|_{\rm per}(Q) := |Du|(Q)
+ \int_{\partial_0 Q} |u^Q(y)-u^Q(\sigma(y))| \,d\mathcal{H}^{n-1}(y).
\end{equation} 
The space $BV_{\rm per}(Q)$ is the space $BV(Q)$ endowed with the norm 
$$\|u\|_{BV_{\rm per}(Q)}:=\|u\|_{L^1(Q)} + |Du|_{\rm per}(Q).$$ Observe that $BV_{\rm per}(Q)$ coincides with $BV(\mathbb T^n)$, 
where $\mathbb T^n:=\R^n/\mathbb Z^n$ is the $n$-dimensional torus. 
For every $E\subseteq Q$ we define the periodic perimeter of $E$ as 
\begin{equation}\label{perper}
\Per(E,\mathbb T^n) :=|D\chi_E|_{\rm per}(Q) 
\end{equation} 
where $\chi_E$ is the characteristic function of $E$. We recall the  isoperimetric inequality \cite{AFP}:    
\begin{proposition}
There exists $C_n>0$ such that 
\begin{equation}\label{iso} 
{\rm{Per}}(E ,\mathbb T^n)\geq C_n|E|^\frac{n-1}{n}
\end{equation}  
for all $E\subseteq Q$ of finite perimeter and such that $|E|\le 1/2$. 
\end{proposition}

\begin{remark}\rm 
Notice that $C_1=2$.  
\end{remark}


In this paper we always make the following regularity assumption on the initial datum and on the forcing term:  
\begin{equation}\label{initialcondition}
u_0, \,g \text{ are Lipschitz continuous and $[0,1]^n$-periodic.}
\end{equation}  
Using the comparison principle\,\cite{bbbl} and \eqref{initialcondition}, we get that there exists a unique   continuous  solution $u$ to \eqref{graphequation} 
with periodic boundary conditions. Moreover, this solution is locally Lipschitz continuous \cite{ck,eh} and hence smooth for all positive times, 
due to the regularity theory for parabolic problems.

\begin{theorem}\label{thex}
Under assumption \eqref{initialcondition}, 
problem \eqref{graphequation} admits a unique solution  
\[u\in \mc C([0,+\infty)\times Q)\cap \mc \mc C^{1+\frac{\alpha}{2},2+\alpha}((0,T]\times Q)\]  for every $\alpha\in (0,1)$ and $T>0$,
with periodic boundary conditions on $\partial Q$. 
Moreover  
$$u_t \in L^2([0,+\infty)\times Q)\quad \text{and}\quad  Du(t,x)\in L^{\infty}([0,T]\times Q)\quad \text{for every}\ T>0. $$
\end{theorem}

We need another condition on the  forcing term $g$, in order to prove existence of generalized traveling wave 
solutions to \eqref{graphequation},
namely we assume that 
\begin{equation}\label{gcondition} \exists A\subseteq Q \text{ such that }   
  \int_{A} g(y)\,dy\, > \,\Per(A,\mathbb T^n).    
\end{equation} 
Note that condition \eqref{gcondition} implies $\max_Q g>0$, 
and is fullfilled for instance if $\int_Q g >0$. 

\begin{remark}\rm 
In \cite{bcn} (see also \cite{ct}) we considered a sort of complementary condition to \eqref{gcondition}. Indeed it is proved  that, if $g$ has zero average and there exists $\delta\in (0,1)$ such that 
\begin{equation}\label{gcondstat}
 \int_{A} g(y)\,dy <\delta \, \Per(A,\mathbb T^n) \qquad  \forall A\subseteq Q ,
\end{equation} 
then there exists a periodic stationary solution of \eqref{graphequation}. 
\end{remark} 

We conclude this section by recalling a classical result about the regularity of hypersurfaces of prescribed 
bounded mean curvature \cite[Thm. 4.1]{m}, \cite[Thm. 1]{t}.
\begin{theorem}\label{regularity}
Let $K$ be a Caccioppoli set with bounded  prescribed mean curvature $A(x)\in L^{\infty}$,  $x\in \partial  K$. Then $\mathcal{H}^{k}(\partial K\setminus  \partial^\star K)=0$   for every $k>n-8$,  and  there exists  $\delta>0$,   such that for every $x\in \partial^\star K$ we get that  $\partial K\cap B(x, \delta)=\partial^\star K\cap B(x, \delta)$ and $ \partial K\cap B(x, \delta)$ is a  $\mathcal{C}^{1+\alpha}$ hypersurface 
for any $\alpha\in (0,1)$.
Moreover, letting $(K_n)_n $ be a sequence of Caccioppoli sets such that:
\begin{itemize}
\item[i)] every $ K_n $  is a locally minimizer  of the functional ${\rm Per}(V) +\int_V  A_n(y)dy $, with $\|A_n\|_\infty\leq A$ independent of $n$,
\item[ii)] $K_n$ converges to $K_\infty$ locally in the $L^1$-topology,
\end{itemize}
and letting $x_n\in \partial K_n$, with $x_n\to x_\infty$ as $n\to+\infty$, we have $x_\infty\in \partial K_\infty$. If $x_\infty\in \partial^\star K_\infty$, then $x_n\in \partial^\star K_n$ for all $n>n_0$, and the unit outward normal to $\partial^\star K_n$ at $x_n$ converges to the unit outward normal to $\partial^\star K_\infty$ at $x_\infty$. 
\end{theorem}

%
%
\section{Existence and regularity of generalized traveling waves}\label{secex}

We now show existence of special solutions to \eqref{graphequation},
which we call {\em generalized traveling waves}. 
They are solutions of the form $\psi(x)+\oc t$, where the graph of $\psi$ is called the profile of the traveling wave and $\oc$ is called the traveling speed.  
Observe that to prove the existence of a traveling wave solution it is sufficient to  determine $c\in \R$ such that the equation 
\begin{equation}\label{equ1} 
- {\rm div}\left( \frac{D\psi }{\sqrt{1+|D\psi|^2}}\right) =  g(y)-   \frac{c}{\sqrt{1+|D\psi |^2}}
\end{equation}
admits a $\Z^n$-periodic solution $\psi:\R^n\to \R$. 
In the following we will show that it is always possible to define a unique  traveling speed $\oc$ for the problem  under our assumption \eqref{gcondition} on the forcing, but in general, the previous equation does not admit a global solution. 
We will prove  that there exists a maximal set $E\subseteq Q$, which is a sufficiently regular domain,  and a function $\psi:Q\to [-\infty, +\infty)$ (which is defined up to additive constants) such that  $E=\{\psi>-\infty\}$,   $\psi\in \mathcal{C}^{2+\alpha}(E)$ and solves 
\begin{equation}\label{equ2} 
- {\rm div}\left( \frac{D\psi }{\sqrt{1+|D\psi|^2}}\right) =  g(y)-  \frac{\oc}{\sqrt{1+|D\psi |^2}},\qquad \text{ in }E 
\end{equation}  
with the boundary conditions
\begin{equation}\label{singularbc}
 \psi(x)\to -\infty\quad \text{  as } {\rm dist}(x,\partial E)\to 0 
\qquad {\rm for\ }\mathcal H^{n-1}-{\rm a.e.\ }x\in\partial E.
\end{equation}  
Moreover we will show that the solutions we construct  satisfy also 
a  stronger boundary condition,  more natural in viscosity solutions theory, say \begin{equation}\label{singularbc2}
\text{ for every $\phi\in \mathcal{C}^1_{\rm per}(\overline{Q})$,\, $\phi-\psi$ achieves its minimum in $E$.}\end{equation} 

First of all we note that the equation \eqref{equ1} can be interpreted, for any $c>0$,  as the Euler-Lagrange equation associated to the functional 
\begin{equation}\label{functionalu}
F_c (\psi)=\int_Q e^{c \psi(y)}\left(\sqrt{1+ |D\psi(y)| ^2} - \frac{g(y)}{c} \right)dy \qquad \psi\in \mc C^1_{\rm per}(Q).
\end{equation}
Using the  change of variable $\Psi (y):= \frac{e^{c\psi(y)}}{c} $, we can rewrite the functional $F_c$ as     
\begin{equation}\label{functionalv}
F_c (\psi)=G_c (\Psi):=\int_Q \sqrt{c^2\Psi^2(y)+ |D\Psi(y) | ^2} -  g(y)\Psi(y) dy, 
\end{equation}
which can extended as a lower semicontinuous functional on $BV_{\rm per}(Q)$, see \cite{AFP}.
Using $G_c$, we can extend the functional $F_c$ to all measurable functions $\psi:Q\to [-\infty,0)$
such that $e^{c\psi(y)}\in BV_{\rm per}(Q)$ (where we use the notation $e^{-\infty}=0$) by setting 
\begin{equation}\label{gf}
F_c(\psi) := G_c\left( \frac{e^{c\psi(y)}}{c}\right).
\end{equation}
In particular, for all such $\psi$ the following representation formula holds (cfr. \cite[Sec. 12]{giustibook}):
\begin{eqnarray}
\nonumber 
F_c (\psi) \!\!&=&\!\! \sup \left\{ \int_Q e^{c \psi(y)}\left( \frac{{\rm div}\phi'}{c}+\phi_{n+1}\right)dy:
\,(\phi',\phi_{n+1})\in \mc C^1_{\rm per}(Q;\R^{n+1}),
\, |\phi'|^2+\phi_{n+1}^2\le 1\right\}
\\ \label{represent}
&& - \int_Q \frac{e^{c \psi(y)}}{c} g(y)\,dy
\end{eqnarray}
which can be easily checked on smooth functions, and then extends by relaxation 
to all $\psi$ such that $e^{c\psi(y)}\in BV_{\rm per}(Q)$.

\begin{proposition} \label{minima1}
Under the standing assumption \eqref{gcondition} there exists a unique   constant $\oc>0$, with $\int_Q g \leq \overline{c}\leq\max _Q g$, such that 
\begin{itemize}
\item if $0<c<\overline{c}$, then   $\inf \{ G_c(\Psi)\ |\   \Psi\in BV_{\rm per}(Q), \ \Psi\geq 0\}=-\infty$,
\item if $c>\overline{c}$, then       $\inf \{ G_c(\Psi)\ |\   \Psi\in BV_{\rm per}(Q), \ \Psi\geq 0\}=0$, and $G_c(\Psi)>0$ for every  $\Psi\not \equiv 0$,
\item    $\min \{ G_\oc(\Psi)\ |\   \Psi\in BV_{\rm per}(Q), \ \Psi\geq 0\}=0$, and there exists $\Psi\not \equiv 0$ s.t. $G_{\oc}(\Psi)=0$.
\end{itemize} 
\end{proposition}

\begin{proof} As $G_c$ is positively one-homogeneous, 
it follows that $\inf_{ \Psi\in BV_{\rm per}(Q), \ \Psi\geq 0} G_c(\Psi)$ can be either $0$ or $-\infty$. 
By definition of $G_c$,  if $c>\max_Q g$, then  $ G_c(\Psi)\geq 0$ for every $\Psi\geq 0$, so 
that $\inf_{\Psi\ge 0} G_c(\Psi)=0$. On the other hand, take  $\Psi=\chi_A$, where $\chi_A$ is the characteristic function 
of the set $A$ appearing in   \eqref{gcondition}.  If $A\subset Q$, then by condition \eqref{gcondition} there exists $k>1$ such that 
\[
G_c(\chi_A)= \Per(A,\T^n)+c|A|-\int_A g< -(k-1)\Per(A,\T^n)+c|A|.  
\]  
Then, choosing $0<c<(k-1)\Per(A,\T^n)/|A|$, we obtain that $G_c( \chi_A)<0$, which implies $ \inf_{\Psi\geq 0} G_c(\Psi)=-\infty$. 
Moreover if  $\int_Q g>0$, then 
\begin{equation}\label{eqqq}
G_c(\chi_Q)<0 \qquad {\rm for\ every\ }0<c<\int_Q\, g . 
\end{equation}
For $c>0$ we consider  the constrained problem 
\begin{equation}\label{constraint} 
\inf\left\{G_c (\Psi) \ |\   \Psi\in BV_{\rm per}(Q), \ \Psi\geq 0,\ \int_Q g\Psi=1\right\}.
\end{equation} 
By the direct method of the Calculus of Variations, one can easily show that 
this problem admits a (possibly nonunique) minimizer $\Psi_c$ \cite{giustibook}. 
We  define the function  minimum value as
\[c\mapsto \mu_c:= G_c(\Psi_c)\] 
and we claim that this function is continuous and strictly increasing. 
Notice that, by minimality of $\Psi_c$, we have
\begin{equation}\label{fici}
\int_Q c\Psi_c dy \leq G_c(\Psi_c) + \int_Q g\Psi_c dy = \mu_c + 1.
\end{equation}
The monotonicity  of $\mu_c$ is due to the fact that $G_c(\Psi_c)$ is increasing as a function of $c$. 
To prove the continuity, we follow  the same argument as in \cite[Prop. 4.1]{nm2}. For $c_1 <c_2$, we get
\begin{eqnarray*}
0 & < &  G_{c_2}(\Psi_{c_2})-G_{c_1}(\Psi_{c_1})
\leq  G_{c_2}(\Psi_{c_1})-G_{c_1}(\Psi_{c_1})
\\
&=& \int_Q \sqrt{{c_2}^2 \Psi_{c_1}^2+|D\Psi_{c_1}|^2}-\sqrt{{c_1}^2 \Psi_{c_1}^2+|D\Psi_{c_1}|^2} 
\\
&=&({c_2}-{c_1}) \int_Q \frac{c(y)\Psi_{c_1}^2}{\sqrt{c(y)^2 \Psi_{c_1}^2+|D\Psi_{c_1}|^2} }dy
\\
&\leq& ({c_2}-{c_1})\int_Q \Psi_{c_1} dy
 \leq  \frac{{c_2}-{c_1}}{c_1}\, \left(\mu_{c_1}+1\right)
\end{eqnarray*}
for some $c(y)\in [c_1,c_2]$, where the last inequality follows from \eqref{fici}.
Since the value function is continuous and strictly increasing, it is possible to define $\oc>0$ as the unique constant for which  
$\mu_\oc= G_{\oc}( \Psi_{\oc}) =0$. From \eqref{eqqq} it follows $\oc\ge \int_Q g$.

Observe  that, due to the constraints,  $\Psi_{\oc}\not\equiv 0$  and,  due to the positive one-homogeneity of $G_c$, 
$k\Psi_{\oc}$ is also a minimizers of $G_{\oc}$ for every $k\geq 0$.  

Finally, observe that necessarily if $c>\oc$ and $\Psi\not\equiv 0$, then $G_c(\Psi)>0$. On the contrary, if $G_c(\Psi)=0$
and   $\Psi\not\equiv 0$, then $\int_Q g \Psi =\lambda>0$. So  $\lambda^{-1}\Psi$ would be a minimizer  to \eqref{constraint},  and $\mu_c=0$, for $c>\oc$, in contradiction with the monotonicity of the value function. 
 \end{proof}

Recalling \eqref{gf}, it is immediate to state the analogous result for the functional $F_c$. 

\begin{corollary}\label{coro}
There exists a unique constant $\oc>0$  with $\int_Q g\leq \overline{c}\leq\max _Q g$ such that 
\begin{itemize}
 \item if $0<c<\overline{c}$, then $\inf \{ F_c(\psi)\ |\   e^{c\psi} \in BV_{\rm per}(Q)\}=-\infty$,
\item if $c> \overline{c}$, then $\inf \{ F_c(\psi)\ |\    e^{c\psi}\in BV_{\rm per}(Q)   \}=0$, and $F_c(\psi)>0$ for all $\psi \not \equiv -\infty$, 
\item  there exists $\psi:Q\to [-\infty, +\infty)$ such that $\psi\not \equiv -\infty$, $e^{\oc \psi}\in BV_{\rm per}(Q)$ and $F_{\oc}(\psi)=0$.
\end{itemize} 
\end{corollary}
\begin{remark}\rm
Notice that Proposition \ref{minima1} and Corollary \ref{coro}, assuring the existence of generalized traveling waves solutions, 
 requires only $g\in L^\infty(Q)$.
\end{remark}
We now analyze the regularity of the minima of $F_\oc$ (or equivalently of $G_\oc$). 

We first give a geometric representation of the functional $F_c$ (cfr. \cite[Thm. 14.6]{giustibook}).
Given $c>0$ and $\Sigma\subset Q\times\R$ we define a weighted perimeter 
\begin{eqnarray} \label{perc}
\Per_c(\Sigma,\T^n\times \R):= 
\sup \left\{ \int_\Sigma e^{c z}\left( {\rm div}\phi(y,z)+c\phi_{n+1}(y,z)\right)dydz:\ \right. \\ \left. 
\phi\in \mc C^1_{\rm per}(Q\times\R;\R^{n+1}),
\, |\phi|^2\le 1\right\}.\nonumber
\end{eqnarray}  
Notice that, for all $\Sigma\subset Q\times\R$ of locally finite perimeter we have 
\[
\Per_c(\Sigma,\T^n\times \R) =
\int_{\partial^*\Sigma} e^{cz}\, d\mathcal H^n 
+ \int_\R e^{ct} \int_{\partial_0 Q} |\chi_\Sigma^Q(y)-\chi_\Sigma^Q(\sigma(y))| \,d\mathcal{H}^{n-1}(y)\,dt
\]
where $\sigma$ is as in \eqref{bvper}.

\begin{proposition}\label{psirepr} 
Let $\psi:Q\to [-\infty, +\infty)$ be such that $e^{c\psi}\in BV_{\rm per}(Q)$. Then
\begin{equation}\label{fcpc}
F_c(\psi) =  \mathcal F_c(\Sigma_\psi):=
\Per_c(\Sigma_\psi,\T^n\times \R)-\int_{\Sigma_\psi} e^{c z} g(y) \, dydz
\end{equation}
where $\Sigma_\psi:=\{(y,z)\in Q\times \R\ |\ z< \psi(y)\}$ is the epigraph of $\psi$.
\end{proposition}

\begin{proof} By exploiting formula \eqref{represent} and the definition of $\Per_c$ in \eqref{perc}, 
it is possible to check  that $ F_c(\psi) \leq \mathcal F_c(\Sigma_\psi)$. For the reverse inequality, we observe first of all that 
\eqref{fcpc}  holds  on smooth functions $\psi\in \mc C^1_{\rm per}(Q)$ and then  the inequality extends to all $\psi$'s by relaxation.
For a similar argument see \cite[Thm. 14.6]{giustibook}.
\end{proof}

\begin{lemma}\label{lemper}
Let $\psi:Q\to [-\infty, +\infty)$ be  a non trivial minimizer of $F_\oc$, then the epigraph 
$\Sigma_\psi$ of $\psi$ is a minimizer, under compact perturbations, 
of the functional $\mathcal F_\oc$ defined in \eqref{fcpc}. 
\end{lemma}

\begin{proof}
We reason as in \cite[Thm. 14.9]{giustibook}.
Given $F\subset Q\times \R$ such that $\int_F e^{\oc z}dydz<+\infty$, we consider
$\psi_F:Q\to [-\infty,+\infty)$ be such that 
\[
\frac{e^{\oc \psi_F(y)}}{\oc} = \int_{-\infty}^{\psi_F(y)}e^{\oc z}dz
=\int_{F_y}e^{\oc z}dz \qquad {\rm for\ a.e.\ }y\in Q,
\]
where $F_y:=\{ z\in\R:\, (y,z)\in F\}$. Observe that, by   definition,  $e^{\oc\psi_F}\in BV(Q)$ and  
\begin{equation}\label{eqg}
\int_{F} e^{\oc z} g(y) \, dydz = \int_Q e^{\oc \psi_F(y)} \frac{g(y)}{\oc} \, dy.
\end{equation}
Moreover, by definition of $\Per_\oc$, for all $\phi=(\phi',\phi_{n+1})\in \mc C^1_{\rm per}(Q;\R^{n+1})$
we have
\begin{equation}\label{eqper}
\Per_\oc(F,\T^n\times \R) \ge 
\int_F e^{\oc z}\left( {\rm div}\phi'+\oc\phi_{n+1}\right)dydz = 
\int_Q e^{\oc \psi_F}\left( \frac{{\rm div}\phi'}{\oc}+\phi_{n+1}\right)dy.
\end{equation}
By taking the supremum over all $\phi$'s in \eqref{eqper},  and using the representation formula \eqref{represent} and  \eqref{eqg}, we then get 
\[
\mathcal F_\oc(F) \ge F_\oc(\psi_F) \ge F_\oc(\psi) = \mathcal F_\oc(\Sigma_\psi)
\]
where the last equality follows from Proposition \ref{psirepr}, thus proving the claim.
\end{proof}

Notice that if $\Sigma$ is a minimizer of $\mathcal F_\oc$, then $\Sigma + (0,z)$ is also a minimizer for all $z\in\R$, that is, 
the class of minimizers is invariant by vertical shifts.
Reasoning as in \cite[Prop. 5.14]{giustibook} (see also \cite{Amb97})
one can prove a density estimate for minimizers of $\mathcal F_\oc$.

\begin{lemma}
There exist constants $\lambda,r_0>0$, depending only on $n$ and $\|g\|_\infty$, such that 
for all minimizers $\Sigma$ of $\mathcal F_\oc$, $x\in\Sigma$ and $r\in (0,r_0)$ the following density estimate holds:
\begin{equation}\label{eqdens}
|\Sigma\cap B_r(x)|\ge \lambda\, r^{n+1}.
\end{equation}
\end{lemma}

\begin{proposition}\label{psireg}
Let $\psi:Q\to [-\infty, +\infty)$ be  a non trivial minimizer of $F_\oc$.
Then $\Gamma_\psi:=\partial\Sigma_\psi$ is a $\mathcal{C}^{2+\alpha}$ hypersurface for all $\alpha<1$, out of a closed singular set 
$S_\psi\subset \Gamma_\psi$
of  Hausdorff dimension  at most   $n-7$.
Moreover, letting $E_\psi:=\Pi_{\R^n}(\Gamma_\psi\setminus S_\psi)$ the projection onto $\R^n$ of $\Gamma_\psi\setminus S_\psi$, we have 
that 
\begin{enumerate}
\item $E_\psi$ is a open set and $E_\psi = {\rm int}(\overline{E_\psi})={\rm int}(\Pi_{\R^n}\Gamma_\psi)$,
\item $\psi\equiv -\infty$ a.e. on $Q\setminus E_\psi$,
\item $\psi\in \mathcal{C}^{2+\alpha}_{\rm loc}(E_\psi)$ for all $\alpha<1$,
\item $\psi$ solves  \eqref{equ2} in $E_\psi$ with boundary conditions \eqref{singularbc}.
\end{enumerate}
Finally, letting $\tilde\psi$ another minimizer of $F_\oc$, 
for every connected component $E_i$ of $E_\psi$ there exists $k_i\in\R$ such that $\tilde\psi= \psi+k_i$.
\end{proposition} 

\begin{proof}
By Lemma \ref{lemper} $\Sigma_\psi$ is a minimizer of $\mathcal F_\oc$ under compact perturbations. 
Classical results about regularity of minimal surfaces with prescribed curvature
\cite{m,Amb97} then imply that $\Gamma_\psi$
is $\mathcal{C}^{2+\alpha}$ for all $\alpha<1$, out of a closed singular set $S_\psi$
of Hausdorff dimension at most  $n-7$.

Recalling that $g$ is Lipschitz continuous and $\Per_\oc(\Sigma_\psi,\T^n\times \R)<+\infty$, 
we can 
reason as in \cite[p. 168 and Prop. 14.11]{giustibook} (see also \cite{giusti})
to obtain that $\nu_{n+1}\ne 0$ on $\Gamma_\psi\setminus S_\psi$, 
where $\nu=(\nu_1,\ldots,\nu_{n+1})$ denotes the exterior unit normal to $\Sigma_\psi$.
Reasoning as in \cite[Thm. 14.13]{giustibook} it then follows 
$E_\psi = {\rm int}(\overline{E_\psi})={\rm int}(\Pi_{\R^n}\Gamma_\psi)$ and 
$\psi\in \mathcal{C}^{2+\alpha}_{\rm loc}(E_\psi)$.
{}From the density estimate  \eqref{eqdens} we can derive that
$\psi\le C$ for some $C>0$, using the same argument as in Thm 14.10, \cite{giustibook}. 
So,  this implies that $\psi$ solves  \eqref{equ2} in $E_\psi$ with boundary conditions \eqref{singularbc}.

To prove the last assertion we notice that, letting 
$\tilde\psi$ be another minimum of $F_\oc$, by convexity we have
$F_\oc (\lambda \psi+(1-\lambda)\tilde\psi)=0$ for every $\lambda \in [0,1]$. 
By definition of $F_\oc$ we then get
\[
0=F_\oc\left(\lambda \psi+( 1-\lambda)\tilde\psi\right)= 
\lambda F_\oc(\psi)+(1-\lambda)F_\oc\left(\tilde\psi\right)
\] 
if and only if 
\[
\psi \,D\tilde\psi =\tilde\psi \,D\psi \qquad {\rm on\ }E_\psi\cap E_{\tilde\psi},
\]
which implies the assertion.
\end{proof}

\begin{remark}\rm
Integrating \eqref{equ2} on $E_\psi$ and using \eqref{singularbc} we obtain 
\begin{equation}\label{per} 
{\rm{Per}}(E_\psi,\T^n) = -\int_{E_\psi} {\rm{div }} 
\left(\frac{D\psi}{\sqrt{1+|D\psi|^2}}\right)dy = 
\int_{E_\psi}  \left( g(y) -\frac{\oc}{\sqrt{1+|D\psi (y)|^2}}\right)dy ,
\end{equation} 
which implies that $E_\psi$ has finite perimeter.
\end{remark}
\begin{corollary}\label{corobc}
Let $\psi$ as in Proposition \ref{psireg}. Then $\psi$ satisfies  the boundary conditions \eqref{singularbc2} on $\partial E_\psi$.
\end{corollary}

\begin{proof}
Let $\phi\in \mathcal{C}^1_{\rm per}(Q)$. 
By Proposition \ref{psireg}, $\min_{\overline{Q}} (\phi-\psi)= \min_{\overline E_\psi } (\phi-\psi)$. 
Assume  by contradiction that $\phi-\psi$ attains its minimum at $y_0\in \partial E_\psi$. 
Without loss of generality, we can assume that $z_0:=\phi(y_0)=\psi(y_0)$ and that $\phi(y)-\psi(y)>0$ for every $y\neq y_0$. 
Again by Proposition \ref{psireg}, we have $x_0:=(y_0, z_0)\in S_\psi$, where $S_\psi$ is the singular set of $\Gamma_\psi$.
 
Let us now blow-up the sets $\Sigma_\psi$ and the subgraph $\Sigma_\phi$ of $\phi$ around $x_0$
If we let 
\begin{eqnarray*}
\Sigma_\psi^s&:=&\{x\in\R^{n+1}|\, sx\in \Sigma_\psi- x_0\}
\\ 
\Sigma_\phi^s&:=&\{x\in\R^{n+1}|\, sx\in \Sigma_\phi- x_0\},
\end{eqnarray*}
by standard arguments of the theory of minimal surfaces \cite[Chapter 9]{giustibook}, 
one can prove that along a subsequence $s_i\to 0$, $\Sigma_\phi^{s_i}$ converges to a half-space $H\subset\R^{n+1}$, 
and $\Sigma_\psi^{s_i}$ converges to a minimal cone $C$. 
{}From the inclusion $\Sigma_\psi\subseteq\Sigma_\phi$ it follows
$C\subseteq H$, but this implies that $C=H$, thus leading to a contradiction 
since the cone $C$ is singular.
\end{proof} 

We now define the maximal support $E$ for minima of the functional $F_\oc$, and  study the regularity of such set.

\begin{proposition} \label{support} 
There exists a  set $E=\cup_{i=1}^k E_i\subseteq Q$, where $E_i$ are connected components, 
such that the support of every minimum $\psi$ of $F_\oc$ is given by the union of some connected components of $E$.   

In particular, if $E$ is connected, then there exists a unique nontrivial minimizer $\psi$ of $F_\oc$, up to an additive constant.

Moreover, there exists a closed set $S\subset\partial E$ such that $\partial E\setminus S$ is a 
$\mathcal{C}^{2+\alpha}$ hypersurface, with $\mathcal{H}^{\gamma} (S)=0$ for every $\gamma>n-8$, and satisfies
the geometric equation
\begin{equation}\label{eqgeomk}
\kappa = g   \qquad \text{\rm on\ }\partial E\setminus S.
\end{equation}

\end{proposition}

\begin{proof} Let  $\psi_1, \psi_2$ be two minima of the functional $F_\oc$ and $E_1, E_2$ be the respective supports. By Proposition \ref{psireg},  if $E_1^i$ and $E_2^j$ are   connected components   respectively of $E_1$ and $E_2$ then either $ E_1^i\cap E_2^j=\emptyset $ or $E_1^i=E_2^j$. In this case there exists a constant $k$ such that $\psi_1=\psi_2+k $ on $E_1^i=E_2^j$. 
We then define $E$ as the union of all the connected components of the supports of the minima of the functional $F_\oc$. 

We claim  that the connected components of $E$ are finite. 
Fix    $E_i$   connected component  of $E$ and  $\psi_i$ solution to \eqref{equ2} with support   $E_i$. 
{}From \eqref{per} we obtain that Per$(E_i,\T^n)\leq \max_Q g|E_i|$. 
This, combined with the isoperimetric inequality \eqref{iso}, gives    that   $|E_i|\geq \left(C_n/\max_Q g\right)^n$, which implies our claim.  

If $E$ is connected, the uniqueness up to addition of constants of the minimizers is  a consequence of Proposition \ref{psireg}.    

We now show the regularity of $\partial E$. Let $\psi\ge 0$ be a minimizer of $F_\oc$ and assume
without loss of generality that $E=E_\psi$. Since $\psi_\lambda = \psi+\lambda$ is also a minimizer for all $\lambda\in \R$,
from the proof of Proposition \ref{psireg} we know that the subgraphs $\Sigma_\lambda=\{(y,z)\in Q\times \R\ |\ z< \psi_\lambda(y)\}$
(locally) minimize the functional $\mathcal F_\oc$ defined in \eqref{fcpc}, for all $\lambda\in \R$. In particular, since $\Sigma_\lambda\to E\times \R$ locally in 
the $L^1$-topology, as $\lambda\to +\infty$, by compactness of quasi minimizers of the area functional \cite{Amb97} we have that 
$E\times \R$ is also a minimizer of $\mathcal F_\oc$ under compact perturbations. 
The thesis then follows by classical regularity theory
for minimal surfaces with prescribed curvature \cite{m,Amb97}.

\end{proof} 

\begin{remark}\rm \label{rkk}
When $n=1$, \eqref{eqgeomk} reduces to 
\begin{equation*}
g = 0 \qquad \text{\rm on\ }\partial E.
\end{equation*}
In particular, $E\ne Q$ necessarily implies $\min_Q g\le 0$.
\end{remark}
\begin{remark}\rm \label{sub}
Let $\psi: E \to \R$ be a minimizer of $F_\oc$ with maximal support, as in the proof of Proposition \ref{support}, and let $\Psi=\frac{e^{\oc\psi}}{\oc}$
be the corresponding minimizer of $G_\oc$. Since $G_\oc$ is a convex functional on $L^2(Q)$, by the general theory of subdifferentials in \cite{brezis,ACM} 
there exist  a vector field 
$\xi_\Psi=\xi:Q\to \R^n$, with $|\xi|\leq 1$  and div$(\xi )\in L^2(Q)$, and a function $h_\Psi=h:Q\to \R$, with $0\leq h \leq 1$, such that  
\begin{equation}\label{subdifferential} 
\int_Q \left(-\text{div}\,\xi  (y)+\oc h (y)- g(y) \right)(w-\Psi) dy\geq 0\qquad \text{ in }Q,
\end{equation} 
for all $w\in BV_{\text{per}}(Q)$ such that $w\geq 0$.
Moreover, for all $y\in E_\psi$,
\begin{eqnarray*}
h(y) &=& \frac{\oc\Psi(y)}{\sqrt{\oc^2\Psi^2(y)+|D\Psi (y)|^2}}
\\
\xi(y) &=& \frac{D\Psi (y)}{\sqrt{\oc^2\Psi^2(y)+|D\Psi (y)|^2}}.
\end{eqnarray*}
If we apply   inequality \eqref{subdifferential} to   $w=\Psi+\chi_F$, where $F\subseteq Q$ is a set of finite perimeter, we obtain   
\begin{equation}\label{diseqvi}
\Per(F,Q)+  \int_F \left(\oc h(y)-g(y)\right) dy\geq 0\,. 
\end{equation} 

In particular, \eqref{per} and \eqref{diseqvi}  imply that $E$ is a minimum for the functional 
\[\mathcal{G}(F)={\rm Per} (F,\T^n) +\int_V \left(\oc h(y)-g(y)\right) dy
\qquad F\subseteq Q.
\] 
\end{remark}

\begin{remark}\rm
We observe that, if $\psi$ a solution to \eqref{equ1} such that $e^{c\psi}\in BV_{\rm per}(Q)$ for some $c>0$, 
which by regularity amounts to say that $\psi$ is bounded from above, then necessarily $F_c(\psi)=0$ so that 
$c\le \oc$ (see Corollary \ref{coro}). Moreover, if $c<\oc$, the support of $\psi$ is strictly smaller than $Q$.  
This means that our variational method selects the {\it fastest} traveling wave solutions to \eqref{graphequation}
which are bounded from above \cite{mu}.

However, there might exist other traveling wave solutions with $c>\oc$, which are not in $BV_{\rm per}(Q)$ 
(see for instance \cite{mrr}). 
\end{remark}

\subsection{Existence of classical traveling waves}\label{globalsolutionconditions} 

In this subsection we state some  condition on the forcing term $g$ under which equation \eqref{equ2} admits a bounded solution $\psi$ in $Q$. 
This problem can be restated as following: find sufficient conditions on   $g$,  under which the maximal support $E$ defined in Proposition \ref{support} coincides with $Q$. 

\begin{remark} \label{stationary} \rm Observe that a first necessary condition on $g$, under which equation \eqref{equ2}, with $\oc>0$,  admits a bounded solution $\psi$ in $Q$ is that $\int_Q g> 0$. In fact, if $\int_Q g= 0 $ and   $\psi$ is a bounded solution to \eqref{equ1}, then $c=0$. In \cite{bcn} we show that condition \eqref{gcondstat} is sufficient to get the existence of a bounded smooth solution to \eqref{equ1} on $Q$ with $c=0$. Proposition \ref{psireg} shows that this condition is essentially optimal for the existence of \emph{stationary wave solutions}. 
\end{remark} 
We consider a solution $\psi$ to \eqref{equ2} with boundary conditions \eqref{singularbc} and  maximal support $E$. Let  
$\Psi=\frac{e^{\oc \psi} }{\oc}$.   
We recall that by \eqref{per}  
\begin{equation}\label{e}
\Per(E,\T^n)=\int_E \left(g(y)- \oc h(y)\right)dy \le \max_Q g\, |E|,
\end{equation} 
where $h=h_\Psi$ is the function defined in \eqref{subdifferential}, In Remark \ref{sub}.  
Since by \eqref{diseqvi}
\[
\int_Q \oc h(y)-g(y)dy \ge 0,
\]
we also have
\begin{equation}\label{f}  \Per(E,\T^n)\leq \int_{Q\setminus E}   \oc h(y) -g(y)dy. 
\end{equation}

{}From inequality \eqref{f}, recalling   $0\leq h\leq 1$ and that $\int_Q g\leq \oc\leq \max_Qg$,  it follows
\begin{equation}\label{emmo}
\Per(E,\T^n)\le \left( \max_Q g-\min_Q g\right)|Q\setminus E|.
\end{equation}

%
%
%
%
%

Assume now $|Q\setminus E|>0$. Recalling the isoperimetric inequality \eqref{iso}, from \eqref{e} and \eqref{emmo} we get
\begin{eqnarray*}
\left( \max_Q g-\min_Q g\right)\frac{1}{2^\frac 1 n}\ge 
\left( \max_Q g-\min_Q g\right)|Q\setminus E|^\frac{1}{n} \ge C_n \qquad {\rm or}\qquad |Q\setminus E|>\frac 1 2.
\end{eqnarray*}
In particular, if 
\begin{equation}\label{equno}
 \max_Q g-\min_Q g < C_n\, 2^\frac 1 n
\end{equation}
we necessarily have $|E|\le 1/2$ and, from \eqref{e},
\begin{equation}\label{eq2}
\max_Q g \geq \frac{\Per(E,\T^n)}{|E|}\geq C_n |E|^{-\frac 1 n}\geq C_n 2^{\frac 1 n}.
\end{equation}  

If $\min_Q g \leq	0$, then \eqref{equno} implies that , in contradiction with \eqref{eq2}. 

If $\min_Q g>0$,   from \eqref{eq2} we get 
\[
\frac 1 2\ge |E|\ge \left( \frac{C_n}{\max_Q g}\right)^n.
\]
{}From \eqref{emmo} it then follows 
\begin{eqnarray*}
\left( \max_Q g-\min_Q g\right)\left(1- \left( \frac{C_n}{\max_Q g}\right)^{n }\right)  &\ge& 
\left( \max_Q g-\min_Q g\right)(1-|E|)
\\
&\ge& C_n |E|^\frac{n-1}{n}
\\
&\ge& C_n \left( \frac{C_n}{\max_Q g}\right)^{n-1}.
\end{eqnarray*} So if $\min_Q g$, we necessarily   have $E=Q$ if either $ \max_Q g < C_n\, 2^\frac 1 n$ or $ \max_Q g \geq  C_n\, 2^\frac 1 n$
and $\max_Q g-\min_Q g  < C_n \left( \frac{C_n}{\max_Q g}\right)^{n-1}\left(1- \left( \frac{C_n}{\max_Q g}\right)^{n }\right)^{-1}$, 

Collecting the previous results above and recalling Remark \ref{rkk} 
we get the following proposition.

\begin{proposition}\label{classicaltw} Assume that $\int_Q g>0$. Then  equation \eqref{equ2} admits a bounded solution $\psi$ in $Q
 $ if one of the following conditions is verified.  
\begin{itemize}
\item[-] $\min_Q g\leq 0$ and $\max_Q g-\min_Qg< C_n 2^{1/n}$;
\item[-] $g>0$ on $Q$ and $\max_Q g < C_n 2^{1/n}$;
\item[-] $g>0$ on $Q$,    $\max_Q g \geq C_n 2^{1/n}$ and $\max_Q g-\min_Qg< \max_Qg \left(\left(\frac{\max_Qg}{C_n}\right)^n-1\right)^{-1}$;
\item[-] $n=1$ and $g>0$ on $Q$ 
\end{itemize}
where $C_n$ is the isoperimetric constant appearing in \eqref{iso} (and $C_1=2$). 
\end{proposition} 

\begin{remark}\rm Observe that   the assumptions in  the previous Proposition assure the existence of    classical  traveling wave solutions to \eqref{graphequation}, i.e. solutions of the form $\oc t+\psi(x)$, where $\psi$ is a smooth, $\Z^n$-periodic solution to  \eqref{equ2}.  
\end{remark}
\begin{remark}\label{knownest}\rm  
In \cite{ls} Lions and Souganidis  showed  that \eqref{equ2} admits a (periodic) solution over all $Q$
if $g$ does not change sign and satifies the condition 
\[
\exists \theta\in (0,1)\ \text{ s.t. }\min_{x\in Q}\left(\theta g^2(x)-(n-1)^2 |Dg(x)|\right)>0.
\]

In \cite{CLS} Cardaliaguet, Lions and Souganidis  proved  
that, when $n=1$ and $\int_0^1 g(y)dy>0$, the following condition implies the solvability of the cell problem: 
\begin{equation}\label{eqcls}
0\leq \int_0^1 g(y)dy-\min_{z\in [0,1]} g(z)<2 .
\end{equation}
\end{remark} 

\section{Stability and long-time behavior} \label{secconv}

If $u$ is a solution to  \eqref{graphequation}, then $w(t,y)=u(t,y)-\oc t$ is  a solution to 
\begin{equation}\label{equationc}  
 w_t   =  \tr\left[\left(\mathbf{I}-  \frac{ Dw \otimes Dw }{1+|Dw |^2}\right)D^2w  \right]+ g  \sqrt{1+|Dw |^2} -\oc 
\qquad  \text{in } (0, +\infty)\times Q
\end{equation}    
with periodic boundary conditions
and initial datum $ w(0,y)  = u_0(y)$. 
Note that  $w$ is the unique solution to \eqref{equationc}, and it is also a classical solution, see Theorem \ref{thex}.
Standard comparison gives that 
$(\min g -\oc)t-\|u_0\|_\infty\leq w(t,x)\leq (\max g -\oc)t+\|u_0\|_\infty$ for every $t\geq 0$, $x\in\R^n$. 
Moreover, under the assumption \eqref{gcondition},  $w$ is bounded (from below) uniformly in $t$.

\begin{lemma} \label{lemmabound}
Let $w$ be the solution to \eqref{equationc} and $\psi$ be any solution to \eqref{equ2}, then   
\begin{equation} \label{bound} 
w(t,y)-\psi(y)\geq \min_{Q}\, (u_0-\psi)\qquad \forall \ t\geq 0, \ y\in Q.
\end{equation} 

Moreover, if there exists a solution $\psi$ to \eqref{equ2} in $Q$, then there exists a constant $M$, depending only on $\|u_0\|_\infty$  
such that $|w(t,x)|\leq M$ for every $t\geq 0$ and $y\in Q$.
\end{lemma}

\begin{proof}  
We fix a $\psi$ solution to \eqref{equ2}, and let   $E=E_\psi$ (see Proposition \ref{psireg}). 
We recall that by Corollary \ref{corobc}, $\psi$ satisfies the boundary conditions \eqref{singularbc2} on $\partial E_\psi$. 

We shall prove that 
\[m(t):= \min_{x\in Q}(w(t,x)-\psi(x)) \] is nondecreasing in $t$. 
Obviously this is sufficient to prove that  $\min_{x\in \overline{E}}(w(t,x)-\psi(x))$ is nondecreasing in $t$. 
We fix $s\geq 0$   and  observe that  $w(t+s,x)$ is the solution to 
\begin{equation*}
v_t (t,x)   =  \tr\left[\left(\mathbf{I}-  \frac{ Dv \otimes Dv }{1+|Dv |^2}\right)D^2v  \right]+  g\left(x\right)  \sqrt{1+|Dv |^2} -\oc \quad \text{in } (0, +\infty)\times E\end{equation*}  
with initial datum $ v(0,x)  = w(s,x)$, and with boundary conditions $v(t,x)=w(t+s,x)$ on $\partial E$ for all $t\geq 0$. 
Notice that $\psi(y)+\min_{\hat y\in Q}(w(s,\hat y)-\psi(\hat y))$ is a regular (stationary) subsolution to the same problem. 
Moreover by 
Corollary \ref{corobc} we have that $w(t+s , x)- [\psi(x)+\min_{y\in Q}(w(s,y)-\psi(y))]$ can attain its minima only in the interior of $E$. So we can apply   comparison principle arguments (see \cite{bbbl}) to conclude  that $w(t+s,x)-\psi(x)\geq \min_{y\in Q}(w(s,y)-\psi(y))$ for   every $t\geq 0$ and $x\in Q$.

Finally, if there exists a solution $\psi$ to \eqref{equ2} in the whole $Q$, then $\psi(x)+\|u_0\|_\infty+\|\psi\|_\infty$ and $\psi(x)-\|u_0\|_\infty-\|\psi\|_\infty$ are, respectively, a supersolution and a subsolution to \eqref{graphequation} and  we conclude by the  standard comparison principle.  
\end{proof}
\begin{remark}\rm \label{remdecr} 
Note that if there is a solution to \eqref{equ2} in the whole $Q$,   a similar  argument gives that 
\[M(t):= \max_{x\in Q}(w(t,x)-\psi(x)) \] is  nonincreasing in $t$. 
\end{remark}

\begin{lemma}\label{lip} 
Let $w$ be a solution to \eqref{equationc}.
Then for all $\tau>0$ there exists a constant $C>0$, depending on $u_0$, $g$ and $\tau$, such that $\|w_t\| \leq C$ for all $t\ge \tau$.
\end{lemma}
\begin{proof}
Recalling Theorem \ref{thex}, we define 
$$
C:= \left\|\tr\left[\left(\mathbf{I}-  
\frac{ Dw(\tau,\cdot) \otimes Dw(\tau,\cdot) }{1+|Dw(\tau,\cdot)|^2}\right)D^2 w(\tau,\cdot)\right]+  g(x) 
\sqrt{1+|D w(\tau,\cdot)|^2} -\oc \right\|_{L^\infty(Q)}<+\infty
$$ 
Then    $S(t,x)= C t + w(t,\cdot)$ is a supersolution to \eqref{equationc} and $s(t, x)=-Ct +w(t,\cdot)$ is a subsolution for all $t>\tau$.
Then by comparison \cite{bbbl} we obtain $-C t  \leq w(t,x)- w(\tau,x)\leq C t $. 
Moreover, for every fixed $s>\tau$, we get that $w(t,x)+\sup_x |w(s,x)-w(\tau,x)|$ 
and $w(t,x)-\sup_x |w(s,x)-w(\tau,x)|$ are respectively a supersolution and a subsolution to \eqref{equationc} with initial data $w(s,x)$. So, again by comparison, and recalling the previous estimate,  for every $\tau\leq s\leq t$ we obtain
$$
-Cs   \leq w(t+s,x)- w(t,x)\leq  C s.
$$ 
\end{proof} 

The estimate in Lemma \ref{lip} imlies  that, for all $t>0$ the function $w(t,\cdot)$ satisfies in the viscosity sense  
\begin{equation}\label{curvatura}  
 -C-g(x)\leq \text{div}\left(\frac{ Dw(t,x) }{\sqrt{1+|Dw(t,x)|^2}}\right)    \leq C +\oc -g(x)\ \text{in }  \R^n\ \forall t\geq \tau. 
\end{equation} 
So, this gives in particular that the curvature of the graph of $w(t,\cdot)$ is uniformly  bounded with respect to $t\in [\tau,+\infty)$.

\begin{proposition}\label{proreg}
Let $\Gamma_w(t)\subset Q\times \R$ be the graph of $w(t,\cdot)$. Then, for all $\tau>0$, $\Gamma_w(t)$ are 
hypersurfaces of class $\mathcal{C}^{1+\alpha}$, for all $\alpha\in (0,1)$, uniformly in $t\in [\tau,+\infty)$. 
\end{proposition}

\begin{proof}
Assume by contradiction the statement to be false. Then we can find $(x_n,t_n)\in Q\times [0,+\infty)$ such that, 
for all $\rho>0$, the hypersurfaces $\Gamma_w(t_n)\cap B_\rho(x_n,t_n)$ are not uniformly $\mathcal{C}^{1+\alpha}$.
Letting $\widetilde w_n(x) := w(x,t_n)-w(x_n,t_n)$, from \eqref{curvatura} we have that 
\begin{equation}\label{curvtilde}  
-\text{div}\left(\frac{ D\widetilde w_n(x) }{\sqrt{1+|D\widetilde w_n(x)|^2}}\right)= h_n(x),    
\end{equation} 
with $\|h_n\|_\infty\le \C$ for some $\C$ independent of $n$. As a consequence $\widetilde w_n$ is a minimizer of
the prescribed curvature functional
\[
\int_Q \left(\sqrt{1+|Du|^2} - h_n u\right)dy.
\]
By the compactness theorem for quasi minimizers of the perimeter \cite{Amb97} the graphs $\Gamma_{\widetilde w_n}$ of $\widetilde w_n$
converge locally in the $L^1$-topology, up to a subsequence, to a limit hypersurface $\Gamma_\infty$ of class $\mc C^{1+\alpha}$.
We can also assume that $x_n\to x$ for some $x\in Q$, and let $\nu_\infty$ be the normal vector to $\Gamma_\infty$ at $(x,0)$.
However, by Theorem \ref{regularity} there exists $\rho>0$ such that $\Gamma_{\widetilde w_n}\cap B_\rho(x,0)$ and 
$\Gamma_{\infty}\cap B_\rho(x,0)$ can all be written as graphs in the direction given by $\nu_\infty$.
Therefore, by elliptic regularity for minimizers of the prescribed curvature functional \cite{m},
the sets $\Gamma_{\widetilde w_n}\cap B_\rho(x,0)$ are uniformly of class $\mc C^{1+\alpha}$ for all $\alpha\in (0,1)$,
thus leading to a contradiction.
\end{proof}  

The following lemma that will be useful in the following. \begin{lemma} \label{lemmadecreasing} 
Let  $F_\oc(v)= \int_Q e^{\oc v(y)}\left(\sqrt{1+ |Dv(y)| ^2} - \frac{g(y)}{\oc} \right)dy$
the functional defined in \eqref{functionalu}. Then for 
every (smooth) solution $w$ to the equation in \eqref{equationc}, 
\begin{equation}\label{decreasing}  
0\le F_\oc(w(t, \cdot))\le F_\oc(u_0 ) \qquad \text{\rm for all }  t>0.
\end{equation} 
\end{lemma} 
\begin{proof} 
For every solution $w$ to \eqref{equationc}, 
using the definition of the functional $F_\oc$, we get
\begin{eqnarray} \frac{d F_\oc(w(t, \cdot))}{dt} &= & \int_Q e^{\oc w}w_t\left[-\text{div }\left(\frac{Dw}{\sqrt{1+|Dw|^2}}\right) -g +\frac{\oc}{\sqrt{1+|Dw |^2}} \right] 
\nonumber 
\\ &=&-\int_Q \frac{ e^{\oc w}w_t^2}{\sqrt{1+|Dw |^2}}\le 0.
\label{decr}\end{eqnarray}
\end{proof} 
The first result on the asymptotic behavior of the solutions $u$ to \eqref{graphequation} is about the convergence of $\frac{u(t,x)}{t}$ as $t\to +\infty$.
\begin{proposition}\label{meanconv}  
Let  $u$ be  the solution to \eqref{graphequation} and $E$ be the maximal support defined in Proposition \ref{support}. 
Then
\[\lim_{t\to +\infty}\frac{\max_{x\in \R^n} u(t,x)}{t}=  \oc,  
\quad\text{and} \quad\lim_{t\to +\infty}\frac{u(t,x)}{t}=\oc  
\quad  \text{   locally uniformly in }E.
\] Moreover if there exists a bounded solution to \eqref{equ2},  \[ \lim_{t\to +\infty}\frac{u(t,x)}{t}=\oc  
\quad  \text{    uniformly in }\R^n.\]   In particular there exists a constant $C\in\R$ such that 
\begin{equation}\label{log} \min_{Q} u_0(x)\leq   M(t):=\max_Q(u(t,x)-\oc t) \leq C+\frac{\log (1+t)}{\oc}.\end{equation} 
 
\end{proposition} 
\begin{proof}
  
Recall that if the stationary problem \eqref{equ2} has a bounded solution, Lemma \ref{lemmabound} gives an uniform bound on  
$u(t,x)-\oc t$, and then we obtain the result. 

We observe, recalling Lemma \ref{lemmabound}, that to prove the general statement    it is sufficient to prove \eqref{log}. 
 The lower bound on $M$ is an immediate consequence of Lemma \ref{lemmabound}, just by choosing $\psi$ as the maximal nonpositive solution to \eqref{equ2}.  

We define $f(t,x):= \frac{2}{\oc}e^{\frac{\oc w(t,x)}{2}}$, so that 
$f_t^2(t,x)= w_t^2(t,x)e^{ \oc w(t,x) }$. Integrating \eqref{decr} between $0$ and $T$, we obtain
\[
C\geq F_\oc(u_0)\geq F_\oc(u_0)-F_\oc(w(T,\cdot)= \int_0^T\int_Q \frac{f_t^2(t,x)}{\sqrt{1+|Dw(t,x)|^2}} dxdt
\] 
for some constant $C>0$ depending only on $u_0$ and  $g$.

Let
\[
\W(t)=\max_x \fint_{B_\rho(x)} f(t,y) dy,
\]
Given a point $\bar x(t)$ where either $M(t)$ or $\W(t)$ attain the maximum,
thanks to Proposition \ref{proreg}
we can choose  $\rho<2/\oc$, independent of $t$, such that $|Dw(x,t)|\leq 1$ for every $x\in B_\rho(\bar x(t))$. 
Notice that 
\[
\frac{2}{\oc}\,e^\frac{\oc (W(t)-\rho)}{2}
\le \W(t)\le \frac{2}{\oc}\,e^\frac{\oc W(t)}{2}\qquad {\rm for\ all\ }t\ge 0,
\]
so that, in order to prove the second inequality in \eqref{log}, it is enough to show
\begin{equation}\label{tildew}
\W(t)\le C(1+\sqrt{t}).
\end{equation}
 
Given $t\ge 0$ let $\mathcal Z(t)$ be the set of points where $\W(t)$ attains its maximum.
Possibly increasing $C$, and using  the fact that  $|Dw(x,t)|\leq 1$ on  $B_\rho(\bar x(t))$,
from the previous inequality we get 
\begin{eqnarray}\label{maxest1}
\nonumber C  &\geq& \int_0^T\max_{\bar x(t)\in\mathcal Z(t)}\fint_{B_\rho(\bar x(t))} f_t^2(t,x) dx\, dt
\\
&\geq& \int_0^T\left(\max_{\bar x(t)\in\mathcal Z(t)}\fint_{B_\rho(\bar x(t))} f_t (t,x) dx\right)^2 dt
\\  \nonumber
&=& \int_0^T\W'(t)^2 dt
  \geq  \frac{1}{T} \left(\int_0^T \vert \W'(t)\vert\, dt\right)^2 .
\end{eqnarray} 
{}From \eqref{maxest1} we then have
\[
\W(T)\le \W(0)+\int_0^T \vert \W'(t)\vert\, dt\le \W(0)+\sqrt{CT},
\]
which proves \eqref{tildew}.
%
\end{proof}
 
We now prove the  main convergence result, on the stability of our traveling wave solutions.

\begin{theorem}\label{convergence2} 
Let $u(t,x)$ be the unique solution to \eqref{graphequation} with periodic boundary conditions, let $M(t):=\max_Q w(t,y)$, 
and let 
\[
\tilde w(t,x):= w(t,x)-M(t) = u(t,x)-\max_{x\in Q} (u(t,x))\le 0.
\]
Then, for any sequence $t_n\to +\infty$ there exists a subsequence $t_{n_k}$ such that, as $k\to+\infty$, 
\begin{equation}\label{conve1} 
w(t_{n_k},x)  \longrightarrow 
\left\{\begin{array}{ll}
\opsi(x)  & {\rm\ locally\ in\ } \mc C^{1+\alpha}(E_\opsi)
\\
-\infty & {\rm\ locally\ uniformly\ in\ } Q\setminus \overline E_\opsi
\end{array}\right.
\end{equation}  
for all $\alpha\in (0,1)$, where $\opsi$ is a traveling wave solution to \eqref{equ2}.

\end{theorem} 

\begin{proof}
We let 
\[
W(t,y):=\frac{e^{\oc w(t,y)}}{\oc},\qquad \WW(t,y):=\frac{e^{\oc \w(t,y)}}{\oc}= e^{-\oc M(t)} W(t,y)\le \frac{1}{\oc}\,.
\]
Notice that from \eqref{equationc} it follows that $W$ satisfies the equation
\begin{equation}\label{equationC}  
W_t   =  \sqrt{\oc^2W^2+|DW|^2}\left({\rm div}\left(\frac{DW}{\sqrt{\oc^2W^2+|DW|^2}}\right)
+ g\right) -\oc^2 W \qquad  \text{in } (0, +\infty)\times Q.
\end{equation} 
By \eqref{decreasing} and \eqref{log}, for all $t\ge 0$ we have
$$
G_\oc(\WW(t,\cdot)) = F_\oc(\w(t, \cdot))=e^{-\oc M(t)} F_\oc(w(t,\cdot))\leq e^{-\oc(\min_Q u_0)} F_\oc(u_0).
$$ 
In particular,  
\[
\int_Q \sqrt{\oc^2\WW^2(t,y) +|D\WW(t,y)|^2}\,dy = G_\oc(\WW(t,\cdot)) +\int g(y)\WW(t,y)\,dy \le C
\]
for all $t\ge 0$, where $C$ depends only on $u_0$ and $g$.
Hence, up to extracting a subsequence $t_{n_k}$, $\WW(t_{n_k},\cdot)\rightharpoonup W_\infty$ weakly* in $BV_{\rm per}(Q)$, as $k\to +\infty$.
Notice that, as in the previous section, the epigraph of $\tilde w(t,\cdot)$  is, for every $t>0$, a minimizer of the prescribed curvature functional 
\[
\Sigma\mapsto \Per_c(\Sigma,\T^n\times \R)-\int_{\Sigma} e^{c z} \tilde g_t(y) \, dydz 
\] where $\tilde g_t$ is an appropriate bounded function, depending on $t$.  
It therefore satisfies the lower density bound \eqref{eqdens}, which implies $W_\infty\not\equiv 0$.
We claim that 
\begin{equation}\label{eqclaim}
G_\oc(W_\infty)=0.
\end{equation}
We introduce the modified functional, for $t>0$,  
\begin{equation*}
\widetilde G_{\oc, t}(W):= \int_Q \left( \sqrt{\oc^2W^2+|DW|^2} - \tilde g_t \,W\right)dy\end{equation*}  where \[\tilde g_t(y):= g(y)-\frac{W_t(t,y)}{\sqrt{\oc^2W^2(t,y)+|DW(t,y)|^2}}\in L^\infty(Q),\qquad \|\tilde g_t\|_\infty \leq C,\]
with $C$ independent of $t$. Note that from  \eqref{equationC}    it follows that, at every $t>0$, $W(t,\cdot)$ is a  critical points of the functional  $ \widetilde G_{\oc,t}$ and so $\widetilde G_{\oc,t}(W(t,\cdot))=0$. Moreover, also  $\widetilde G_{\oc,t}(\WW(t,\cdot))=0$. 
Recalling \eqref{decr}, up to extracting a further subsequence, we can assume that
\begin{eqnarray}\label{eqalpha}
\partial_t G_\oc(W(t_{n_k},\cdot)) = \partial_t F_\oc(w(t_{n_k},\cdot)) &=&
-\int_Q \frac{e^{\oc w(t_{n_k},y)}w_t^2(t_{n_k},y)}{\sqrt{1+|Dw(t_{n_k},y)|^2}}\,dy
\\ \nonumber
&=& -\int_Q \frac{W_t^2(t_{n_k},y)}{\sqrt{\oc^2W^2(t_{n_k},y)+|DW(t_{n_k},y)|^2}}\,dy \to 0
\end{eqnarray}
as $k\to +\infty$.

Since $G_\oc(v)\ge 0$ for every $v$,  to prove the claim \eqref{eqclaim} it is sufficient to show that $G_\oc(W_\infty)\leq 0$. We get, using the convexity of   $G_{\oc}$ and the definition of the modified functional $\widetilde G_{\oc,t}$, 
\begin{eqnarray*}
G_\oc(W_\infty)&\le& \liminf_{k\to +\infty}G_\oc(\WW(t_{n_k},y))  \\
&=& \liminf_{k\to +\infty}\left( \widetilde G_{\oc, t_{n_k}}(\WW(t_{n_k},y))
-\int_Q\frac{\WW(t_{n_k},y) W_t(t_{n_k},y)}{\sqrt{\oc^2W^2(t_{n_k},y)+|DW(t_{n_k},y)|^2}}\,dy\right) \\
&=& \liminf_{k\to +\infty} -\int_Q\frac{\WW(t_{n_k},y) W_t(t_{n_k},y)}{\sqrt{\oc^2W^2(t_{n_k},y)+|DW(t_{n_k},y)|^2}}\,dy  \end{eqnarray*}
 since $\widetilde G_\oc(\WW(t_{n_k},y))=0$. 
 Using the H\"older inequality, \eqref{eqalpha} and the definition of $\WW$, we obtain \begin{eqnarray*}  &&  \liminf_{k\to +\infty}   \int_Q\frac{-\WW(t_{n_k},y) W_t(t_{n_k},y)}{\sqrt{\oc^2W^2(t_{n_k},y)+|DW(t_{n_k},y)|^2}}\,dy \\ &\leq&   
 \liminf_{k\to +\infty}   
 \left(\int_Q \frac{W^2_t(t_{n_k},y)}{\sqrt{\oc^2W^2(t_{n_k},y)+|DW(t_{n_k},y)|^2}}\,dy \right)^\frac{1}{2}
 \left(\int_Q \frac{e^{-\oc M(t)}}{\oc^2}dy\right)^\frac{1}{2} =0
 \end{eqnarray*} which proves our claim. 
 In particular, $\opsi:=\log(\oc W_\infty)/\oc: E_\opsi\to [-\infty,+\infty)$
is a traveling wave solution of \eqref{equ2} with $c=\oc$.

Let us now prove \eqref{conve1}. Given $y\in E_\opsi$, by Theorem \ref{regularity}
there exists $r>0$ such that $B_r(y)\subset E_\opsi$ and $\|D\tilde w(t_{n_k},y)\|_{L^\infty(B_r(y))}$ is uniformly bounded in $k$.
By standard elliptic regularity \cite{GT} it then follows that the functions $\tilde w(t_{n_k},\cdot)$ are uniformly bounded 
in $\mc C^{1+\alpha}(B_r(y))$ for all $\alpha\in (0,1)$, so that they converge to $\opsi$ locally in $\mc C^{1+\alpha}(E_\opsi)$.

Fix now $y\in Q\setminus\overline E_\opsi$ and take $r>0$ such that $B_r(y)\subset Q\setminus\overline E_\opsi$.
Assume by contradiction that there exist $c\in\R$ and $y_k\in B_r(y)$, $k\in\mathbb N$, such that
$\tilde w(t_{n_k},y_k)\ge c$ for all $k$. By the density estimate \eqref{eqdens} this would imply
$\int_Q \WW(t_{n_k},y)dy\ge c'$ for some $c'\in\R$, contradicting the fact that 
$\WW(t_{n_k},y)\to W_\infty$ in $L^1(Q)$, with $W_\infty\equiv 0$ in $B_r(y)$.
We thus proved \eqref{conve1}.
\end{proof} 
\begin{remark}\rm 
If the functional $F_\oc$ admits a unique minimizer $\bar \psi:E_\opsi\to \R$ up to an additive constant 
(for instance if the maximal support $E$ is connected, see Proposition \ref{support}), then instead of \eqref{conve1} we have
\begin{equation}\label{conve2} 
\lim_{t\to +\infty}w(t,x) \,=\,  
\left\{\begin{array}{ll}
\opsi(x)-\max^{}_{\overline E_\opsi} \bar\psi  & {\rm\ locally\ in\ } \mc C^{1+\alpha}(E_\opsi)
\\
-\infty & {\rm\ locally\ uniformly\ in\ } Q\setminus \overline E_\opsi
\end{array}\right.
\end{equation}  
for all $\alpha\in (0,1)$.
\end{remark}
\begin{corollary}\label{convfort}
Let $u(t,x)$ be the unique solution to \eqref{graphequation} with periodic boundary conditions, and assume that there exist bounded solutions to \eqref{equ2} in $Q$ (see Proposition \ref{classicaltw}). Then 
\begin{equation*}  
u(t,x)-\oc t \longrightarrow\opsi(x) \qquad {\rm in\ } \mc C^{1+\alpha}(Q),{\rm\ as\ }t\to+\infty,
\end{equation*} where $\opsi$ is a bounded solution to \eqref{equ2}.
\end{corollary}
\begin{proof} 
By Lemma \ref{lemmabound} and Remark \ref{remdecr}, it is enough to prove that $w(t_n,x)\to \opsi(x) $ uniformly along a subsequence $t_n\to +\infty$. 
This result can be obtained by repeating the same argument as in the proof of Theorem \ref{convergence2}. \end{proof} 

\begin{remark}\rm 
A straightforward adaptation of the argument in Corollary \ref{convfort} gives that, under assumption \eqref{gcondstat},  
\[ 
u(t,x) \to \psi(x) \qquad {\rm in\ } \mc C^{1+\alpha}(Q),{\rm\ as\ }t\to+\infty, 
\] 
where $\psi$ is a stationary solution  of the parabolic equation \eqref{graphequation} (whose existence has been shown in \cite{bcn}).
\end{remark}

\begin{remark}\rm 
The results of this paper can be easily extended to equation \eqref{graphequation} considered 
on a bounded open set $\Om\subset\R^n$ with Lipschitz boundary, and with Neumann boundary conditions on $\partial \Om$.
\end{remark}


\end{document}